\documentclass[12pt]{article}
\usepackage[utf8]{inputenc}%          source encoding
\usepackage[T1]{fontenc}%
\usepackage[a4paper]{geometry}
\usepackage{amsfonts}
\usepackage{amsthm}
\usepackage{amsmath}
\usepackage{amssymb,mathtools}
\usepackage{enumerate}
\usepackage{wasysym}
\usepackage{hyperref,microtype}
\hypersetup{pdfstartview=XYZ}%        default zoom

\small\normalsize

\newcommand{\littleo}[1]{o\mathopen{}\left(#1\right)}
\newcommand{\ome}[1]{\Omega\mathopen{}\left(#1\right)}
\newcommand{\abs}[1]{\left\lvert#1\right\rvert}
\newcommand{\dyck}{\mathrm{DYCK}}
\newcommand{\col}{\mathrm{COL}}
\newcommand{\sst}[2]{\left\{#1\,:\,#2\right\}}
\renewcommand{\le}{\leqslant}
\renewcommand{\ge}{\geqslant}

\allowdisplaybreaks[1]

\newtheorem{theorem}{Theorem}
\newtheorem{lemma}[theorem]{Lemma}
\newtheorem{claim}{Assertion}
\newtheorem{corollary}[claim]{Corollary}

\def\C{{\mathcal C}}
\def\W{{\mathcal W}}

\def\O{{\mathcal O}}
\def\R{{\mathcal R}}

\def\FinalConstant{\frac9{2^{5/3}}}
\def\FinalConstantNum{{2.83483}}
\def\KappaConstant{\frac{2^{5/3}}3}

\begin{document}
\title{A note on acyclic vertex-colorings} \date{}

\author{Jean-Sébastien Sereni
  \thanks{\emph{Centre National de la Recherche
  Scientifique} (LORIA), Vandœuvre-lès-Nancy, France. E-mail: \texttt{sereni@kam.mff.cuni.cz}.
  This author's work was partially
  supported by the French \emph{Agence Nationale de la Recherche} under reference
  \textsc{anr 10 jcjc 0204 01}.
  }
 \and Jan Volec
  \thanks{Mathematics Institute and DIMAP, University of Warwick, Coventry CV4 7AL, UK. E-mail:
  \texttt{honza@ucw.cz}. This author's work was supported by a grant of the French Government.}
}

\maketitle
\begin{abstract}
  \noindent We prove that the acyclic chromatic number of a graph
  with maximum degree $\Delta$ is less than $2.835\Delta^{4/3}+\Delta$.
  This improves the previous upper bound, which
  was $50\Delta^{4/3}$. To do so, we draw inspiration from
  works by Alon, McDiarmid and Reed and by Esperet and Parreau.
\end{abstract}

\section{Introduction}
In 1973, Gr\"ubaum~\cite{Gru73} considered proper colorings of graphs with an additional
constraint: the subgraph induced by every pair of color classes is required
to be acyclic. Such colourings are coined \emph{acyclic colorings} and
the least integer $k$ such that a graph $G$ admits an acyclic coloring with
$k$ colors is the \emph{acyclic chromatic number $\chi_a(G)$} of $G$.

Three years later, Erd\H{o}s (see~\cite{AlBer76}) raised the question of determining
the maximum possible value of $\chi_a(G)$ over all graphs $G$ with maximum
degree $\Delta$. Let $\chi_a(\Delta)$ be this value.
A first indication is given by the following observation:
for every graph $G$, any proper coloring of $G^2$ is an acyclic coloring of
$G$. Therefore, $\chi_a(\Delta)\le\Delta^2+1$.
However, Erd\H{o}s conjectured a stronger statement, namely that
$\chi_a(\Delta)=\littleo{\Delta^2}$ as $\Delta$ tends to infinity.

This conjecture was confirmed about a quarter century later, by Alon,
McDiarmid and Reed~\cite{AMR91}. Relying on the Lov\'asz Local
Lemma~\cite{ErLo75}, they established the following upper bound.
\begin{theorem}[Alon, McDiarmid \& Reed~\cite{AMR91}]
   For every positive integer $\Delta$,
   \[
   \chi_a(\Delta)\le50\Delta^{4/3}.
   \]
\end{theorem}

This upper bound, more than confirming Erd\H{o}s's conjecture, turns out
to be of order very close to that of $\chi_a(\Delta)$.
Indeed, Alon, McDiarmid and Reed~\cite{AMR91} further proved that
\[\chi_a(\Delta)=\ome{\frac{\Delta^{4/3}}{(\log\Delta)^{1/3}}}.\]

Our goal is to exploit the recent advances regarding algorithmic versions of
the Local Lemma, inspired by the incompressibility arguments. In 2009
Moser~\cite{Mos09} and, in 2010, Moser and
Tardos~\cite{MoTa10} designed strong algorithmic versions of the
Local Lemma. More importantly for our purposes, while preparing
his talk for the Symposium on Theory of Computing, Moser found
a simpler proof of his result from 2009. The technique
used in this proof became known as the ``entropy compression'' argument;
the reader is referred to Fortnow's website~\cite{For} and Tao's blog~\cite{Tao} for more details.

Independently, Schweitzer~\cite{Sch09} pursued a similar line of research,
explaining how to obtain constructive bounds on van der Warden numbers. His
work was subsequently improved by Kulich and Keme\v{n}ov\'a~\cite{KuKe11} to
precisely match the known non-constructive results.

All these ideas inspired new adaptations and more efficient uses of the essence of the Local Lemma
to tackle various combinatorial questions, in particular graph
colouring problems~\cite{DJKW,EsPa13,GKM13, Prz12, PrzSchSk13} and problems
related to pattern avoidance~\cite{OchePin13}.
We draw inspiration from the original work of Alon, McDiarmid \&
Reed~\cite{AMR91} and a recent result of Esperet \& Parreau~\cite{EsPa13} to
establish the following upper bound.
\[
\chi_a(\Delta) \le \FinalConstant\cdot\Delta^{4/3} + \Delta <
\FinalConstantNum \cdot \Delta^{4/3} + \Delta.
\]

\section{Proof of the Upper Bound}
We shall use certain standard estimates on the number of Dyck words with all
descent of even lengths. A \emph{partial Dyck word} is a bit string
$w$ such that no prefix of $w$ contains more ones than zeros. A \emph{Dyck
word} is a partial Dyck word of length $2t$ with exactly $t$ zeros. A
\emph{descent} in a partial Dyck word is a maximal sequence of consecutive
ones.

The following lemma is a special case of~\cite[Lemmas 7 and 8]{EsPa13} for Dyck words
with all descents of even length.
It follows from~a folklore bijection between Dyck words and
plane trees, and the asymptotic results for counting such trees; see, e.g., \cite[Theorem 5]{Drmo04}.
More details are found in the work of Esperet and Parreau~\cite{EsPa13}.
\begin{lemma}
\label{lem:dyck}
There exists an absolute constant $C_{\dyck}$ such that the number of Dyck words of
length $2t$ with all descents of even length is at most
\[
C_\dyck\cdot \frac{\left(3\sqrt{3}/2\right)^t}{t^{3/2}}.
\]
\end{lemma}
\noindent
We also recall a special case of~\cite[Lemma 6]{EsPa13}.
\begin{lemma}
\label{lem:p_dyck}
Let $r$ be a non-negative integer. The number of partial Dyck words with
exactly $t$ zeros, exactly $(t-r)$ ones, and all descents of even length is at most
\[
C_\dyck\cdot \frac{\left(3\sqrt{3}/2\right)^{t+r}}{(t+r)^{3/2}}.
\]
\end{lemma}

We are now ready to present our main result.
\begin{theorem}\label{th-main}
Fix a positive integer $\Delta$ and a real $\kappa$ such that $\kappa \ge 2/\Delta^{2/3}$.
If $G$ is a graph with maximum degree $\Delta$,
then the acyclic chromatic number $\chi_a(G)$ is at most
\[
f(\Delta,\kappa) \coloneqq \left(\frac1\kappa +
\frac32\sqrt{\frac{3\kappa}{2}} \right)\Delta^{4/3} + \Delta -
\frac{\Delta^{1/3}}{\kappa}.
\]
In particular, if $\Delta \ge 3$ and $\kappa=\KappaConstant$, it follows that
$\chi_a(G) \le \FinalConstant\cdot\Delta^{4/3} + \Delta < \FinalConstantNum \cdot \Delta^{4/3} + \Delta$.
\end{theorem}
\begin{proof}
Fix a graph $G$ with maximum degree $\Delta$.
Without loss of generality, let $V(G) = \{1,\ldots,n\}$.
The main idea of the proof is as follows.

We first consider a randomized procedure that takes as input a partial acyclic
coloring of $G$ using $f(\Delta,\kappa)$ colors and tries to assign a random
color from a specifically restricted subset of $f(\Delta,\kappa)$ colors to the
smallest (with respect to its number) uncolored vertex $v$. If the partial
coloring extended by the coloring of $v$ is still a partial acyclic coloring of
$G$, then the procedure ends --- and thus this extended partial coloring is
kept. On the other hand, if the coloring
of $v$ creates a two-colored cycle, or if $v$ is assigned the same
color as one of its neighbors, then the procedure uncolors a specific subset of colored vertices
(which includes $v$) and then ends. This procedure is called \textbf{EXTEND}.

Next, we set up a procedure \textbf{LOG} that creates a compact record containing enough
information to be able to perform the following. Suppose we have a partial acyclic
coloring $c$ of $G$ with $f(\Delta,\kappa)$ colors. We execute EXTEND and obtain a new
partial acyclic coloring $c'$ of $G$.  Furthermore, let $x$ be the (randomly chosen)
color that EXTEND tried to assign to the smallest uncolored vertex $v$ in $c$.
The record constructed by LOG shall contain enough information that it is possible
to reconstruct both $c$ and $x$ from the record and $c'$.
Our aim is to create the record in such a way that, in an amortized sense, its
size is smaller than that of the list that EXTEND can choose the color $x$ from.

Finally, we consider the following randomized coloring algorithm. Start with an
empty coloring, that is, every vertex is uncolored in the initial partial
coloring. Then repeatedly execute the procedures EXTEND and LOG
until all the vertices of $G$
are assigned a color
in the current partial coloring.
One execution of EXTEND followed by one execution of LOG is called a
\emph{step} of the algorithm.

Note that the algorithm might never terminate. However, we show that the
probability that it actually does terminate, after sufficiently many steps, is
positive.  This will follow from the fact that after $t$ steps (for a
sufficiently large integer $t$), the number of ways how to $t$-times choose a
color in the procedure EXTEND will be (strictly) greater than the number of
all possible records corresponding to the executions that have not terminated
in $t$ steps times the number of all possible precolorings (recall our aim to
make the amortized size of a record small).

Let us now be precise.
For a vertex $v \in V(G)$, let $D(v)$ be the set of vertices $u \in V(G)$
different from $v$ such that the number of common neighbors of $u$ and $v$
is at least $\kappa \cdot \Delta^{2/3}$.  By symmetry, $u \in D(v) \iff v \in D(u)$.
A vertex $u \in D(v)$ is said to be \emph{dangerous for $v$}.

If $u$ and $v$ are dangerous for each other, then there are lots of
$4$-cycles containing both $u$ and $v$, namely $\Omega(\Delta^{4/3})$. This is why
the procedure EXTEND is designed in such a way that it never tries to assign to $v$ a
color that is currently assigned to a vertex that is dangerous for $v$.
Similarly, the procedure shall never try to assign to $v$ a color that is
currently assigned to one of the neighbors of $v$.
Formally, for a partial acyclic coloring $c$, we let
\begin{itemize}
\item $c[N(v)]$ be the set of colors assigned in $c$ to the neighbors of $v$;
\item $c[D(v)]$ be the set of colors assigned in $c$ to the vertices that are
dangerous for $v$; and
\item $L_c(v) \coloneqq \{1,2,\dots,f(\Delta,\kappa)\} \setminus \big(c[N(v)] \cup c[D(v)] \big)$.
\end{itemize}
Note that
$\abs{c[N(v)]} \le \Delta$. Moreover,
$\abs{c[D(v)]} \le \left({\Delta^{4/3}} - {\Delta^{1/3}}\right)/\kappa$.
Indeed, since the number of edges $\{w,w'\}$
with $w\in N(v)$ and $w'\in V(G)\setminus\{v\}$ is at most $\Delta(\Delta -1)$,
the size of $D(v)$ is at most
$\left({\Delta^{4/3}} - {\Delta^{1/3}}\right)/\kappa$.

Therefore, \[\abs{L_c(v)} \ge \frac32\sqrt{\frac{3\kappa}{2}}\cdot\Delta^{4/3}.\]
For the simplicity of our analysis, we shall always assume that
$\abs{L_c(v)} = \frac32\sqrt{\frac{3\kappa}{2}}\cdot\Delta^{4/3}$ (in the case of
having a strict inequality for some choice of $c$ and $v$, we simply remove
$\abs{L_c(v)} - \frac32\sqrt{\frac{3\kappa}{2}}\cdot\Delta^{4/3}$ colors from
$L_c(v)$ arbitrarily).

Next, for a vertex $v$ and an integer $k$, we give an upper bound on the number of $2k$-cycles
incident with $v$ that could become two-colored at some step of the execution
of the algorithm.
\begin{claim}
\label{cl:cyclecount}
For a vertex $v \in V(G)$ and an integer $k\ge2$, let $\C_{2k}(v)$ be the set of all
$2k$-cycles $W=v,w_2,w_3,\dots,w_{2k}$ incident with $v$ such that no two vertices at distance two on
$W$ are dangerous for each other. Then
\begin{equation*}
\abs{\C_{2k}(v)} < \left( \Delta^{4/3} \cdot \sqrt{\kappa/2} \right)^{2k-2}.
\end{equation*}
\end{claim}
\begin{proof}
We actually show that
\begin{equation*}
\abs{\C_{2k}(v)} < \frac\kappa2 \cdot \Delta^{2k-4/3}.
\end{equation*}
Since $\kappa \ge 2/\Delta^{2/3}$ and $k\ge2$, we have $\frac\kappa2 \cdot
\Delta^{2k-4/3} \le \left( \Delta^{4/3} \cdot \sqrt{\kappa/2} \right)^{2k-2}$
and the statement then follows.
First, there are at most $\binom\Delta2 < \Delta^2/2$ choices of $w_2$ and $w_{2k}$. Fix a choice
of $w_2$ and $w_{2k}$. Next, we fix one by one the vertices $w_3,
w_4,\dots,w_{2k-2}$; for each of them,
there are at most $\Delta -1 < \Delta$ choices.
Finally, since
$w_{2k-2}$ and $w_{2k}$ are not dangerous for each other, there are less
than $\kappa \cdot \Delta^{2/3}$ choices to choose
$w_{2k-1}$.
Combining all estimates together, we conclude that
\[
\abs{\C_{2k}(v)} < \frac\kappa2 \cdot \Delta^{2+2k-4+2/3} = \frac\kappa2 \cdot \Delta^{2k-4/3}.
\]
\end{proof}

The last bit that we need to describe the procedure EXTEND is to fix linear
orderings on the $2k$-cycles in $\C_{2k}(v)$ for every $v\in V(G)$ and
$k\ge2$. Fix $v$ and $k$, and
consider a $2k$-cycle $v,w_2,w_3,\dots, w_{2k}$ containing $v$. We define the
identifier of the cycle as follows: if $w_2 < w_{2k}$, then the identifier is
$w_2w_3\dots w_{2k}$; otherwise,
it is $w_{2k}w_{2k-1}\dots w_2$. The linear ordering $\O_{2k}(v)$ of the
elements of $\C_{2k}(v)$ is just given by the lexicographical ordering of
their identifiers.

Now we are ready to describe the procedure EXTEND. It takes as input a partial
acyclic coloring $c$, and outputs a new partial acyclic coloring $c'$. The
procedure is defined as follows.
\begin{itemize}
\item Let $v$ be the smallest uncolored vertex in $c$.

\item Pick a color $x$ uniformly at random from the list $L_c(v)$.

\item If the extension of $c$ obtained by assigning the color $x$ to $v$
is a partial acyclic coloring of $G$,
then we set $c'$ to be this extension.

\item Otherwise, let $\W$ be the set of all two-colored cycles in the extension of $c$.
Let $W\in\W$ be the $2k$-cycle that has the largest length and, subject
to that, the lexicographically smallest identifier $w_2w_3\dots w_{2k}$.
We set $c'$ to be the restriction of $c$ to the vertices $V\setminus\{w_4,w_5,\dots,w_{2k}\}$,
i.e., we uncolor the vertex set of $W$ except the two adjacent vertices $w_2$ and $w_3$.
\end{itemize}

We continue with the description of the procedure LOG. At the end of its
$t$-th execution, LOG outputs a record $R^t$ that is based on the previous record $R^{t-1}$ and
the coloring and possible uncolorings that happened during the $t$-th execution of
EXTEND. In order to make the analysis easier, we decompose $R^t$ into two parts $R^t_1$ and $R^t_2$
and analyse them separately. A record $R^t_1$ shall be a bit string that keeps track of all
colorings and uncolorings that have been performed during the first $t$ executions of EXTEND,
and a record $R^t_2$ shall be an integer that stores the information about the
$2k$-cycles that have been uncolored.

We thus define $R^t_1$ and $R^t_2$ recursively. For convenience, we let $R_{1}^{0}$ be the
empty string and $R_{2}^{0}\coloneqq0$. Now assume that $t\ge1$.
Let $v$ be the smallest uncolored vertex after the $(t-1)$-th execution of
EXTEND, so $v=1$ if $t=1$.
If the $t$-th execution of EXTEND assigns a color to
$v$ and keeps the extended colouring,
then we set $R^t_1$ to be $R^{t-1}_1$ to which we append one $0$, and $R^t_2 \coloneqq R^{t-1}_2$.
Otherwise, let $W$ be the $2k$-cycle uncolored during the $t$-th execution of EXTEND, and
let $z$ be the index of $W$ in $\C_{2k}(v)$ ordered according to $\O_{2k}(v)$.
Recall that $z$ is always an integer between $1$ and
$\max\sst{\abs{\C_{2k}(v)}}{v\in V(G)}$, which is at most
$\left\lfloor\left( \Delta^{4/3} \cdot \sqrt{\kappa/2}\right)^{2k-2}\right\rfloor$.
We let $R^t_1$ be $R^{t-1}_1$ to which we append one $0$ and $(2k-2)$ ones,
and we set
\[
R^t_2 \coloneqq R^{t-1}_2 \cdot \left\lfloor\left( \Delta^{4/3} \cdot \sqrt{\kappa/2}\right)^{2k-2}\right\rfloor + (z-1).
\]

Let us realize that the records $R^{t-1}_1$ and $R^{t-1}_2$ can be
reconstructed from the records $R^t_1$ and $R^t_2$.  Indeed, let $p$ be the
position of the last $0$ in $R^t_1$ and let $q$ be the number of ones after
this $0$, noting that $q$ might be equal to zero.  Then $R^{t-1}_1$ is equal
to the first $p-1$ elements of $R^t_1$ and $R^{t-1}_2$ is equal to
\[
\Bigg\lfloor {R^t_2 \Big/ \left\lfloor\left( \Delta^{4/3} \cdot
\sqrt{\kappa/2}\right)^{q}\right\rfloor} \Bigg\rfloor.
\]

Our next step is to show that the records $R^t_1$ and $R^t_2$ actually also contain
enough information to determine the set of uncolored vertices after $t$ steps
of the algorithm.
\begin{claim}\label{cl-prev}
For any positive integer $t$, the records
$R^t_1$ and $R^t_2$ determine the set $V_t$, defined to be the set
of uncolored vertices of $G$ after $t$ steps of the algorithm.
\end{claim}
\begin{proof}
We prove the statement by induction on the positive integer $t$.
If $t=1$, then necessarily $R^1_1$ is the list
containing only one zero, $R^1_2 = 0$, and $V_t = \{2,3,\dots,n\}$. Suppose now
that $t>1$.
As we observed above, $R^t_1$ and $R^t_2$ determine the records $R^{t-1}_1$
and $R^{t-1}_2$.
By the induction hypothesis, $R^{t-1}_1$ and $R^{t-1}_2$ determine $V_{t-1}$.
Therefore, we can find the smallest vertex $v$ in $V_{t-1}$,
which is the vertex that EXTEND attempts to color in the $t$-th step.

If $R^t_1$ is equal to $R^{t-1}_1$ with one $0$ appended, then coloring $v$
has not created any two-colored cycle and hence $V_t = V_{t-1}\setminus
\{v\}$. On the other hand, if $R^t_1$ is equal to $R^{t-1}_1$ with one $0$ and
$q$ ones appended, where $q\ge1$, then we set $z\coloneqq\left(R^t_2 \bmod
\left\lfloor\left( \Delta^{4/3} \cdot \sqrt{\kappa/2}\right)^{q}\right\rfloor
\right) + 1$ and we let $w_2w_3\dots w_{q+2}$ be the identifier of the $z$-th
element of $\C_{q+2}(v)$ according to $\O_{q+2}(v)$. Since this was the
$(q+2)$-cycle that was uncolored during the $t$-th execution of EXTEND, we deduce
that $V_t = V_{t-1}\setminus \{w_4,w_5,\dots,w_{q+2}\}$.
\end{proof}

Finally, we show that the records $R^t_1$ and $R^t_2$ together with the partial
coloring after $t$ steps fully determine the partial coloring after $t-1$
steps of the algorithm.
\begin{claim}
\label{cl:reconstruct}
Fix a positive integer $t$.
Let $c$ be the partial coloring of $G$ obtained after $t-1$ steps of
the algorithm, $c'$ the partial
coloring after $t$ steps, and $x$ the color that was used to color
the smallest uncolored vertex during the $t$-th execution of EXTEND.
Then $R^t_1$, $R^t_2$ and $c'$ determine both $x$ and $c$.
\end{claim}
\begin{proof}
Again, we prove the assertion by induction on the positive integer $t$.
If $t=1$, then $c'$ contains exactly
one colored vertex. Its color is $x$ and $c$ is indeed the empty coloring.

Let $t>1$. We first use $R^t_1$ and $R^t_2$ to determine the records
$R^{t-1}_1$ and $R^{t-1}_2$. Next, we utilize Assertion~\ref{cl-prev} and, using
$R^{t-1}_1$ and $R^{t-1}_2$, we determine the smallest uncolored vertex $v$
after the $(t-1)$-th step of the algorithm. Now, as in the proof of
Assertion~\ref{cl-prev},
the records $R^{t-1}_1, R^{t-1}_2, R^t_1$ and $R^t_2$ are used
to determine if the coloring of $v$ at the $t$-th execution of EXTEND has
created a two-colored cycle or not. In the former case, we also determine,
again in the same way as in the proof of Assertion~\ref{cl-prev}, the identifier $w_2w_3\dots
w_{2k}$ of the two-colored $2k$-cycle incident with $v$ that was uncolored by
EXTEND.

If there was no two-colored cycle, then clearly $x=c'(v)$ and $c$ can be
obtained from $c'$ by uncoloring the vertex $v$. On the other hand, if EXTEND
uncolored the $2k$-cycle with the identifier $w_2w_3\dots w_{2k}$, then we know that
$x=c'(w_3)$ and $c$ can be obtained by modifying $c'$ in the following way: we
color the vertices $w_4, w_6,\dots, w_{2k}$ with the color $c'(w_2)$, and the
vertices $w_5,w_7,\dots,w_{2k-1}$ with the color $c'(w_3)$.
\end{proof}

Before we continue the exposition and present our upper bounds on
the number of possible records that the procedure LOG can create, let us introduce
some additional notation. Again, we consider the situation just after
$t$ steps of the algorithm. For an integer $i\le t$, let $u_i$ be the number of
vertex-uncolorings that were performed during the $i$-th execution of EXTEND.
Specifically, if the coloring that was performed at the $i$-th execution did not
create any two-colored cycle, then $u_i=0$. On the other hand, if during this
execution EXTEND uncolored a two-colored $2k$-cycle, then $u_i=2k-2$. Next,
let $U_i \coloneqq \sum_{j=1}^t u_j$, that is, $U_i$ is the total number of vertex-uncolorings
that were performed from the beginning of the first step till the end of the
$i$-th step. Since
each execution of EXTEND performs exactly one vertex-coloring, it follows that $U_i \le
i$ for every $i\le t$. (In fact, one even sees that $U_i < i$.)

We are now ready to present the following two assertions which, assuming that the
algorithm has not colored the whole graph after $t$ steps, give upper bounds on
the number of possible records $R^t_1$ and $R^t_2$, respectively.
\begin{claim}
\label{cl:R1}
Let $\R^t_1$ be the set of all possible records $R^t_1$ that can be obtained by
performing $t$ steps of the algorithm that do not result in coloring the whole
graph $G$. Then there exists an absolute constant
$C$, depending only on $G$ and not on $t$, such that
\[
\abs{\R^t_1} \le C \cdot\frac{\left(3\sqrt{3}/2\right)^t}{t^{3/2}}.
\]
\end{claim}
\begin{proof}
Let $c$ be the partial coloring of $G$ obtained after $t$ steps of the
algorithm.
Assume that $c$ is not an acyclic coloring of the whole graph $G$.

By its definition, the record $R^t_1$ contains exactly $t$ zeros, and
for each $i\le t$, the $i$-th zero is followed by exactly $u_i$ ones.
Since $U_i \le i$ for all $i\le t$, the record $R^t_1$ is a partial Dyck word.
Thus the number of $1$'s in $R^t_1$ can be written as $t-r$ for some
non-negative integer $r$.  Further, the difference between the number of $0$'s
and the number of $1$'s in $R^t_1$ is equal to the number of colored vertices
in $c$, hence $r\le n-1$.
Therefore, it follows from Lemma~\ref{lem:p_dyck} that
\[
\abs{\R^t_1} \le \sum_{r=0}^{n-1} C_\dyck \cdot
\frac{\left(3\sqrt{3}/2\right)^{t+r}}{(t+r)^{3/2}}
\le \left( n \cdot C_\dyck \cdot \left(3\sqrt{3}/2\right)^{n-1} \right) \cdot
\frac{\left(3\sqrt{3}/2\right)^t}{t^{3/2}}.
\]
\end{proof}

\begin{claim}
\label{cl:R2}
For any positive integer $t$, the record $R^t_2$ is an integer satisfying
\[
0 \le R^t_2 \le \left( \Delta^{4/3} \cdot \sqrt{\kappa/2} \right)^{U_t} - 1.
\]
\end{claim}
\begin{proof}
We prove the statement by induction on $U_t$.
If $U_t=0$, then $R^t_2=0$. Assume now that $U_t > 0$. Let $i$ be the number of the step
where the $U_t$-th uncoloring occurs. Thus, during the $i$-th step, the
procedure EXTEND attempts to color a vertex $v$, which creates a two-colored
cycle. Let $\ell$ be the length of this cycle and $z$ its index in
$\C_{2k}(v)$ ordered by $\O_{2k}(v)$.
Assertion~\ref{cl:cyclecount} implies that the integer $z$ is at most $\left\lfloor\left(
\Delta^{4/3} \cdot \sqrt{\kappa/2}\right)^{\ell-2}\right\rfloor$.
Moreover, the induction hypothesis ensures that $R_{2}^{i-1}$ is an integer
satisfying
\[
0\le R^{i-1}_2 \le \left( \Delta^{4/3} \cdot \sqrt{\kappa/2}
\right)^{U_t-(\ell-2)} - 1.
\]
The conclusion follows, since
\begin{align*}
R^t_2&=R^{i-1}_2 \cdot \left\lfloor\left( \Delta^{4/3} \cdot
\sqrt{\kappa/2}\right)^{\ell-2}\right\rfloor + (z-1)\\
&\le \left(\Delta^{4/3} \cdot \sqrt{\kappa/2}\right)^{U_t} - \left\lfloor\left(
\Delta^{4/3} \cdot \sqrt{\kappa/2}\right)^{\ell-2}\right\rfloor + (z-1)\\
&\le \left( \Delta^{4/3} \cdot \sqrt{\kappa/2} \right)^{U_t} - 1.
\end{align*}
\end{proof}
\noindent
Since $R^t_2$ is always an integer and $U_t\le t$, we immediately deduce the following.
\begin{corollary}
\label{cor:R2}
For any positive integer $t$, the record
$R^t_2$ is an integer between $0$ and $\left\lfloor \left( \Delta^{4/3} \cdot \sqrt{\kappa/2} \right)^{t}\right\rfloor - 1$.
\end{corollary}

The only thing that remains to do in order to finish the proof of
Theorem~\ref{th-main} is to combine the assertions together. Let $C_\col$ be
the number of all possible partial
acyclic colorings of $G$ using $f(\Delta,\kappa)$ colors. So $C_\col
\le \left(f(\Delta,\kappa)+1\right)^n$.
Therefore, using Assertion~\ref{cl:R1} and Corollary~\ref{cor:R2}, we infer
that there are at most
\[
C_\col \cdot C \cdot \left(\left(3\sqrt{3}/2\right)^t \cdot t^{-3/2}\right) \cdot \left(\Delta^{4/3} \cdot \sqrt{\kappa/2}\right)^t
= \littleo{1} \cdot \left( \frac32\sqrt{\frac{3\kappa}2}\cdot\Delta^{4/3} \right)^t
\]
choices for a tuple $(c',R^t_1,R^t_2)$, where the $\littleo{1}$ term tends to
$0$ as $t$ tends to infinity.
On the other hand, by repeatedly applying Assertion~\ref{cl:reconstruct}, a tuple $(c',R^t_1,R^t_2)$
determines the (randomly chosen) color $x$ at the $i$-th step for every $i\le t$.
Therefore, assuming that the algorithm has not terminated after the $t$-th
step --- that is, there are still some uncolored vertices --- it had at most
$\littleo{1} \cdot \left( \frac32\sqrt{\frac{3\kappa}2}\cdot \Delta^{4/3} \right)^t$
possible ways how to choose the colors from the corresponding lists.
Hence, if $t$ is large enough, the algorithm terminates with a positive
probability --- in fact, this probability tends to $1$ as $t$ tends to infinity.

We conclude that
\[
\chi_a(G) \le f(\Delta,\kappa) = \frac32\sqrt{\frac{3\kappa}2}\Delta^{4/3} +
\left(\Delta^{4/3} - \Delta^{1/3}\right)/\kappa + \Delta,
\]
which finishes the proof.
\end{proof}


\begin{thebibliography}{99}

\bibitem{AlBer76}
{\sc M.~O.~Albertson and D.~M.~Berman}, {\em The acyclic chromatic number},
	Proceedings of the Seventh Southeastern Conference on Combinatorics, Graph
  Theory and Computing, Utilitas Mathematica Inc., Winnipeg, Canada (1976),
  pp.~51--60.

\bibitem{AMR91}
{\sc N.~Alon, C.~McDiarmid, and B.~Reed}, {\em Acyclic coloring of graphs},
  Random Structures Algorithms, 2 (1991), pp.~277--288.

\bibitem{Drmo04}
{\sc M. Drmota}, {\em Combinatorics and Asymptotics on Trees},
  Cubo Journal, 6, No.2 (2004).

\bibitem{DJKW}
{\sc V.~Dujmovi{\'c}, G.~Joret, J.~Kozik, and D.~R. Wood}, {\em Nonrepetitive
  colouring via entropy compression}, Combinatorica, forthcoming.

\bibitem{ErLo75}
{\sc P.~Erd{\H{o}}s and L.~Lov{\'a}sz}, {\em Problems and results on
  {$3$}-chromatic hypergraphs and some related questions}, in Infinite and
  finite sets ({C}olloq., {K}eszthely, 1973; dedicated to {P}. {E}rd{\H o}s on
  his 60th birthday), {V}ol. {II}, North-Holland, Amsterdam, 1975,
  pp.~609--627. Colloq. Math. Soc. J\'anos Bolyai, Vol. 10.

\bibitem{EsPa13}
{\sc L.~Esperet and A.~Parreau}, {\em Acyclic edge-coloring using entropy
  compression}, European J. Combin., 34 (2013), pp.~1019--1027.

\bibitem{For}
{\sc L.~Fortnow}, {\em A Kolmogorov Complexity Proof of the Lov{\'a}sz Local
  Lemma}, June 2009.
\newblock
  \url{http://blog.computationalcomplexity.org/2009/06/kolmogorov-complexity-proof-of-lov.html}.

\bibitem{Gru73}
{\sc B.~Gr{{\"u}}nbaum}, {\em Acyclic colorings of planar graphs}, Israel J.
  Math., 14 (1973), pp.~390--408.

\bibitem{GKM13}
{\sc J.~Grytczuk, J.~Kozik, and P.~Micek}, {\em New approach to nonrepetitive
  sequences}, Random Structures Algorithms, 42 (2013), pp.~214--225.

\bibitem{KuKe11}
{\sc T.~Kulich and M.~Keme{\v{n}}ov{{\'a}}}, {\em On the paper of {P}ascal
  {S}chweitzer concerning similarities between incompressibility methods and
  the {L}ov{\'a}sz local lemma}, Inform. Process. Lett., 111 (2011),
  pp.~436--439.

\bibitem{Mos09}
{\sc R.~A. Moser}, {\em A constructive proof of the {L}ov{\'a}sz local lemma},
  in S{TOC}'09---{P}roceedings of the 2009 {ACM} {I}nternational {S}ymposium on
  {T}heory of {C}omputing, ACM, New York, 2009, pp.~343--350.

\bibitem{MoTa10}
{\sc R.~A.~Moser and G.~Tardos}, {\em A constructive proof of the general
  lov\'asz local lemma}, J. ACM, 57 (2010), pp.~11:1--11:15.

\bibitem{OchePin13}
{\sc P.~Ochem and A.~Pinlou}, {\em Application of entropy compression in pattern avoidance},
  ArXiv e-prints, 1301.1873 (2013).

\bibitem{Prz12}
{\sc J.~Przybyło}, {\em On the facial Thue choice index via entropy compression},
  ArXiv e-prints, 1207.0964 (2012).

\bibitem{PrzSchSk13}
{\sc J.~Przybyło, J.~Schreyer, E.~Škrabuľáková}, {\em  On the facial Thue choice number of plane graphs via entropy compression method},
  ArXiv e-prints, 1308.5128 (2013).

\bibitem{Sch09}
{\sc P.~Schweitzer}, {\em Using the incompressibility method to obtain local
  lemma results for {R}amsey-type problems}, Inform. Process. Lett., 109
  (2009), pp.~229--232.

\bibitem{Tao}
{\sc T.~Tao}, {\em Moser’s entropy compression argument}, August 2009.
\newblock
  \url{http://terrytao.wordpress.com/2009/08/05/mosers-entropy-compression-argument}.
\end{thebibliography}
\end{document}